\documentclass{amsart}
\usepackage{amssymb}
\usepackage{amsmath}
\usepackage{amsfonts}

\newtheorem{theorem}{Theorem}
\theoremstyle{plain}

\newtheorem{corollary}{Corollary}

\newtheorem{definition}{Definition}
\newtheorem{example}{Example}

\newtheorem{lemma}{Lemma}

\newtheorem{proposition}{Proposition}
\newtheorem{remark}{Remark}

\numberwithin{equation}{section}
\input{tcilatex}

\begin{document}
\title[NON-ISOTROPIC POTENTIAL THEORETIC INEQUALITY ]{NON-ISOTROPIC
POTENTIAL THEORETIC INEQUALITY }
\author{M. Esra YILDIRIM}
\address{[Department of Mathematics, Faculty of Science, University of
Cumhuriyet, 58140, Sivas, Turkey}
\email{mesra@cumhuriyet.edu.tr}
\author{Abdullah AKKURT}
\address{[Department of Mathematics, Faculty of Science and Arts, University
of Kahramanmara\c{s} S\"{u}t\c{c}\"{u} \.{I}mam, 46000, Kahramanmara\c{s},
Turkey}
\email{abdullahmat@gmail.com}
\author{H\"{u}seyin YILDIRIM}
\address{[Department of Mathematics, Faculty of Science and Arts, University
of Kahramanmara\c{s} S\"{u}t\c{c}\"{u} \.{I}mam, 46000, Kahramanmara\c{s},
Turkey}
\email{hyildir@ksu.edu.tr}
\date{}
\subjclass[2010]{31B10, 26A33, 35B45, 35B65, 46E30, 43A15, 47B37}
\keywords{Adams trace inequalitity, Stummel class, Morrey spaces,
non-isotropic distance}
\thanks{M.E. Yildirim was partially supported by the Scientific and
Technological Research Council of Turkey (TUBITAK Programme 2228-B)}

\begin{abstract}
In this paper, the new weighted inequalities were derived by $\beta $%
-distance which is similar to the given inequality for the potential
operator defined in \cite{AD71}.
\end{abstract}

\maketitle

\section{Introduction}

The following inequality has been obtained by D. Adams \cite{AD71};

Let $V$ is a non negative function in the Morrey space $L_{1,\lambda }\left( 
\mathbb{R}^{n}\right) ,\lambda >n-p$.

For $\forall u\in C_{0}^{\infty }\left( \mathbb{R}^{n}\right) ,\ q=p\frac{%
\lambda }{n-p},\ 1<p<n$, the following inequality is valid;%
\begin{equation}
\left( \int_{\mathbb{R}^{n}}|u\left( x\right) |^{q}V\left( x\right)
dx\right) ^{\frac{1}{q}}\leq C\left( p,\lambda ,n\right) \left\Vert
V\right\Vert _{L_{1,\lambda }\left( \mathbb{R}^{n}\right) }^{\frac{1}{q}%
}\left\Vert \nabla u\right\Vert _{L^{p}\left( \mathbb{R}^{n}\right) }
\label{0.1}
\end{equation}%
where $L_{1,\lambda }\left( {\mathbb{R}^{n}}\right) $ is\ Morrey space.

Morrey spaces $L_{p,\lambda }$ were introduced by Morrey in 1938 in
connection with certain problems in elliptic partial differential equations
and calculus of variations \cite{Morrey}. Later, Morrey spaces found
important applications to Navier Stokes and Schr\"{o}dinger equations,
elliptic problems with discontinuous coefficients and potential theory. An
exposition of the Morrey spaces can be found in the book \cite{Kuf}.

Morrey spaces were widely studied during last decades, including the study
of classical operators of harmonic analysis such as maximal, singular and
potential operators.

\begin{definition}
Let $1\leq p<\infty $, $0\leq \lambda \leq n$. We define the Morrey space $%
L_{p,\lambda }\left( {\mathbb{R}^{n}}\right) $ as the set of locally
integrable functions $f$ with the finite norms%
\begin{equation}
\Vert f\Vert _{L_{p,\lambda }}:=\sup_{x\in {\mathbb{R}^{n}},t>0}t^{-\frac{%
\lambda }{p}}\Vert f\Vert _{L_{p}(B(x,t))},
\end{equation}
\end{definition}

Note that if $p=1$, $L_{1,\lambda }\left( {\mathbb{R}^{n}}\right) $ Morrey
space is defined as follows;%
\begin{equation}
\begin{array}{ll}
L_{1,\lambda }\left( \mathbb{R}^{n}\right) & =\left\{ f\in L_{1}^{loc}\left( 
\mathbb{R}^{n}\right) :\left\Vert f\right\Vert _{L_{1,\lambda }\left( 
\mathbb{R}^{n}\right) }\right. \\ 
& \left. \equiv \sup_{x\in \mathbb{R}^{n},r>0}\frac{1}{r^{\lambda }}%
\int_{\left\vert x-y\right\vert <r}\left\vert f\left( y\right) \right\vert
dy<+\infty \right\} ,\ 0<\lambda <n.%
\end{array}
\label{0.2}
\end{equation}

According to the definition of $L_{p,\lambda }$, the parameter $p$ describes
the local integrability, while $\lambda $ describes measure the global
integrability. Unlike $L_{p,\lambda }$ with $p>1$, it is not the case that
we can characterize $L_{1,\lambda }$ in terms of the Littlewood-Paley
decomposition. For this reason, the singular integral operators like the
Riesz transforms are not bounded on $L_{1,\lambda }$. Nevertheless , this
space can be compared with other function spaces. This is what we do in the
present paper.

This paper aims at using $\beta $-distance establish an imbedding similar to
(\ref{0.1}), assuming more general hypotheses on the function $V$.

Firstly, we define a non isotropic distance or $\beta $-distance in $n$
dimensional Euclidean space $%
\mathbb{R}
^{n}$.

It is well known that the families of integral operators with positive
kernels have many applications in different problems, in the theory of
differantial equation, harmonic analysis etc. Integral operators depending
on difference between the variables have princibal aplications. For
multidimensional case, this type of kernels are function of euclidean
distance between two points.

Let $\beta =\left( \beta _{1},\beta _{2},\ldots ,\ \beta _{n}\right) ,\beta
_{k}\geq \frac{1}{2},$ $k=1,2,\ldots ,n$ and $\left \vert \beta \right \vert
=\beta _{1}+\beta _{2}+\ldots +\beta _{n}$. For $x=\left( x_{1},\ldots
,x_{n}\right) $ and $y=\left( y_{1},\ldots ,y_{n}\right) $%
\begin{equation*}
\left \vert x-y\right \vert _{\beta }:=(\left \vert x_{1}-y_{1}\right \vert
^{\frac{1}{\beta _{1}}}+\left \vert x_{2}-y_{2}\right \vert ^{\frac{1}{\beta
_{2}}}+\ldots +\left \vert x_{n}-y_{n}\right \vert ^{\frac{1}{\beta _{n}}})^{%
\frac{\left \vert \beta \right \vert }{n}},
\end{equation*}%
is the non-isotropic distance or $\beta $-distance $x$ and $y$, given in 
\cite{BeLi60}, (\cite{SY06} $-$ \cite{OZ06})$,$ \cite{YI05}.

For any positive $t$, it is easy to see that this distance has the following
properties of homogeneity$,$ 
\begin{equation}
\left( \left\vert t^{\beta _{1}}x_{1}\right\vert ^{\frac{1}{\beta _{1}}%
}+\ldots +\left\vert t^{\beta _{n}}x_{n}\right\vert ^{\frac{1}{\beta _{n}}%
}\right) ^{\frac{\left\vert \beta \right\vert }{n}}=t^{\frac{\left\vert
\beta \right\vert }{n}}\left\vert x\right\vert _{\beta },\ t>0.  \label{0.3}
\end{equation}%
This equality give us that non-isotropic $\beta $-distance is the order of a
homogeneous function $\frac{\left\vert \beta \right\vert }{n}$. Thus the
non-isotropic $\beta $-distance has the following properties:%
\begin{equation*}
\begin{array}{l}
1.\left\vert x\right\vert _{\beta }=0\Leftrightarrow x=\theta ,\theta
=\left( 0,0,\ \ldots ,\ 0\right) . \\ 
\\ 
2.\left\vert t^{\beta }x\right\vert _{\beta }=\left\vert t\right\vert ^{%
\frac{\left\vert \beta \right\vert }{n}}\left\vert x\right\vert _{\beta }.
\\ 
\\ 
3.\left\vert x+y\right\vert _{\beta }\leq k(\left\vert x\right\vert _{\beta
}+\left\vert y\right\vert _{\beta }),%
\end{array}%
\end{equation*}

where $k=2^{\left( 1+\frac{1}{\beta _{\min }}\right) ^{\frac{\left \vert
\beta \right \vert }{n}}}$, $\beta _{\min }=\min \left \{ \beta _{1},\ \beta
_{2},\ \ldots ,\ \beta _{n}\right \} $.

Here we consider $\beta $-spherical coordinates by the following formulas:%
\begin{equation}
\begin{array}{l}
x_{1}=\left( \rho \cos \varphi _{1}\right) ^{2\beta _{1}}. \\ 
\\ 
x_{2}=\left( \rho \sin \varphi _{1}\cos \varphi _{2}\right) ^{2\beta _{2}}.
\\ 
\vdots \\ 
x_{n-1}=\left( \rho \sin \varphi _{1}\sin \varphi _{2}\ldots \sin \varphi
_{n-2}\cos \varphi _{n-1}\right) ^{2\beta _{n-1}} \\ 
\\ 
x_{n}=\left( \rho \sin \varphi _{1}\sin \varphi _{2}\ldots \sin \varphi
_{n-2}\sin \varphi _{n-1}\right) ^{2\beta _{n}}%
\end{array}
\label{0.4}
\end{equation}%
where $0\leq \varphi _{1},\varphi _{2},\ldots ,\varphi _{n-2}\leq \pi $ and $%
0\leq \varphi _{n-1}\leq 2\pi $.

By using $\beta $-spherical coordinates, we get that $\left\vert
x\right\vert _{\beta }=\rho ^{\frac{2\left\vert \beta \right\vert }{n}}$.

Firstly, we will define the $\beta -$ball $B_{\beta }\left( x,r\right) $
generated by the $\beta -$distance. For a positive $r$ and any $x\in 
\mathbb{R}
^{n}$, the open $\beta $-ball with radius $r$ and a center $x$ as%
\begin{equation*}
B_{\beta }\left( x,r\right) =\left\{ \sigma :\left\vert x-y\right\vert
_{\beta }<r\right\} .
\end{equation*}

In \cite{RaCaZa01}, $S_{p}$ class has been obtained. Now, we introduce $%
S_{p} $ class depending on $\beta $-distance as follow. Let $1<p<n,$%
\begin{equation}
S_{p}^{\beta }=\left \{ f\in L_{1}^{loc}\left( \mathbb{R}^{n}\right)
:\sup_{x\in \mathbb{R}^{n}}\int_{\left \vert x-y\right \vert _{\beta }<r}%
\frac{\left \vert f\left( y\right) \right \vert }{\left \vert x-y\right
\vert _{\beta }^{\left( n-p\right) \frac{2\left \vert \beta \right \vert }{n}%
}}dy=\eta _{\beta }\left( r\right) \searrow 0\text{ }for\text{ }r\searrow
0\right \} .  \label{0.5}
\end{equation}

\section{PRELIMINARY RESULTS}

In this section, we introduce Morrey space $L_{1,\lambda }^{\beta }\left( 
\mathbb{R}^{n}\right) $ and $S_{p}^{\beta }$, we give some results relating
them. The Stummel class $S_{p}$ was introduced by Ragusa and Zamboni \cite%
{RaCaZa01}. This class is a class of functions related to local behavior of
mapping by generalized fractional integral operators and the generalized
Morrey spaces are classes of functions related to local behavior of
Hardy-Littlewood maximal function.Now, we introduce $S_{p}$ class depending
on $\beta $-distance as follows.

\begin{definition}
Let $1<p<n,$%
\begin{equation}
S_{p}^{\beta }=\left\{ f\in L_{1}^{loc}\left( \mathbb{R}^{n}\right)
:\sup_{x\in \mathbb{R}^{n}}\int_{\left\vert x-y\right\vert _{\beta }<r}\frac{%
\left\vert f\left( y\right) \right\vert }{\left\vert x-y\right\vert _{\beta
}^{\left( n-p\right) \frac{2\left\vert \beta \right\vert }{n}}}dy=\eta
_{\beta }\left( r\right) \searrow 0\text{ }for\text{ }r\searrow 0\right\} .
\label{0.5}
\end{equation}
\end{definition}

$L_{1,\lambda }^{\beta }\left( \mathbb{R}^{n}\right) $ Morrey space is
defined as follows.

\begin{definition}
Morrey space $L_{1,\lambda }^{\beta }\left( \mathbb{R}^{n}\right) $
generated by $\beta $-distance;%
\begin{equation*}
\begin{array}{ll}
L_{1,\lambda }^{\beta }\left( \mathbb{R}^{n}\right) & =\left \{ f\in
L_{1}^{loc}\left( \mathbb{R}^{n}\right) :\left \Vert f\right \Vert
_{L_{1,\lambda }^{\beta }\left( \mathbb{R}^{n}\right) }\right. \\ 
& \left. \equiv \sup_{x\in \mathbb{R}^{n},r>0}\frac{1}{r^{\frac{2\left \vert
\beta \right \vert }{n}\lambda }}\int_{\left \vert x-y\right \vert _{\beta
}<r}\left \vert f\left( y\right) \right \vert dy<+\infty \right \}
,0<\lambda <n%
\end{array}%
\end{equation*}%
where $\lambda >n-p$.
\end{definition}

The next lemma gives a relation between the space $S_{p}^{\beta }$ and $%
L_{1,\lambda }^{\beta }$.

\begin{lemma}
\textit{If} $V$ \textit{belongs to} $L_{1,\lambda }^{\beta }\left( \mathbb{R}%
^{n}\right) $, \textit{then} $V$ \textit{belongs to} $S_{p}^{\beta }$, 
\textit{and} 
\begin{equation*}
\int_{\left \vert x-y\right \vert _{\beta }<r}\frac{\left \vert V\left(
y\right) \right \vert }{\left \vert x-y\right \vert _{\beta }^{\left(
n-p\right) \frac{2\left \vert \beta \right \vert }{n}}}dy\leq C\left(
n,p,\lambda ,\beta \right) r^{\left( \lambda -\left( n-p\right) \right) 
\frac{2\left \vert \beta \right \vert }{n}}\left \Vert V\right \Vert
_{L_{1,\lambda }^{\beta }\left( \mathbb{R}^{n}\right) },
\end{equation*}%
where $\left( n-p\right) \frac{2\left \vert \beta \right \vert }{n}<\lambda 
\frac{2\left \vert \beta \right \vert }{n}<n\frac{2\left \vert \beta
\right
\vert }{n}.$
\end{lemma}

\textit{Conversely, if} $V$ \textit{belongs to} $S_{p}^{\beta }$ \textit{and}
$\eta _{\beta }\left( r\right) \sim r^{\alpha \frac{2\left \vert \beta
\right \vert }{n}}$ \textit{then} $V$ \textit{belongs to} $L_{1,\left(
n-p+\alpha \right) \frac{2\left \vert \beta \right \vert }{n}}^{\beta
}\left( \mathbb{R}^{n}\right) $.

\begin{proof}
About the first part, we have%
\begin{equation*}
\begin{array}{l}
\int_{\left \vert x-y\right \vert _{\beta }<r}\frac{\left \vert V\left(
y\right) \right \vert }{\left \vert x-y\right \vert _{\beta }^{\left(
n-p\right) \frac{2\left \vert \beta \right \vert }{n}}}dy \\ 
\\ 
=\sum \limits_{k=0}^{+\infty }\int \limits_{\frac{r}{2^{k+1}}\leq \left
\vert x-y\right \vert _{\beta }<\frac{r}{2^{k}}}\frac{\left \vert V\left(
y\right) \right \vert }{\left \vert x-y\right \vert _{\beta }^{\left(
n-p\right) \frac{2\left \vert \beta \right \vert }{n}}}dy \\ 
\\ 
\leq \sum \limits_{k=0}^{+\infty }\left( \frac{2^{k+1}}{r}\right) ^{\left(
n-p\right) \frac{2\left \vert \beta \right \vert }{n}}\int \limits_{\left
\vert x-y\right \vert _{\beta }<\frac{r}{2^{k}}}\left \vert V\left( y\right)
\right \vert dy \\ 
\\ 
\leq 2^{\left( n-p\right) \frac{2\left \vert \beta \right \vert }{n}}\sum
\limits_{k=0}^{+\infty }\left( \frac{2^{k}}{r}\right) ^{\left( n-p\right) 
\frac{2\left \vert \beta \right \vert }{n}}\left( \frac{r}{2^{k}}\right) ^{%
\frac{2\left \vert \beta \right \vert }{n}\lambda }\sup \limits_{r>0}\frac{1%
}{\left( \frac{r}{2^{k}}\right) ^{\frac{2\left \vert \beta \right \vert }{n}%
\lambda }}\int \limits_{\left \vert x-y\right \vert _{\beta }<\frac{r}{2^{k}}%
}\left \vert V\left( y\right) \right \vert dy \\ 
\\ 
\leq 2^{\left( n-p\right) \frac{2\left \vert \beta \right \vert }{n}%
}r^{\left( \lambda -\left( n-p\right) \right) \frac{2\left \vert \beta
\right \vert }{n}}\sum \limits_{k=0}^{+\infty }2^{k\frac{2\left \vert \beta
\right \vert }{n}\left( \left( n-p\right) -\lambda \right) }\left \Vert
V\right \Vert _{L_{1,\lambda }^{\beta }\left( \mathbb{R}^{n}\right) } \\ 
\\ 
=r^{\left( \lambda -\left( n-p\right) \right) \frac{2\left \vert \beta
\right \vert }{n}}C\left( n,p,\lambda ,\beta \right) \left \Vert V\right
\Vert _{L_{1,\lambda }^{\beta }\left( \mathbb{R}^{n}\right) }.%
\end{array}%
\end{equation*}%
The second part is obvious, indeed%
\begin{equation*}
\int_{\left \vert x-y\right \vert _{\beta }<r}\left \vert V\left( y\right)
\right \vert dy\leq r^{\left( n-p\right) \frac{2\left \vert \beta \right
\vert }{n}}\int_{\left \vert x-y\right \vert _{\beta }<r}\frac{\left \vert
V\left( y\right) \right \vert }{\left \vert x-y\right \vert _{\beta
}^{\left( n-p\right) \frac{2\left \vert \beta \right \vert }{n}}}dy\leq
Cr^{\left( n-p+\alpha \right) \frac{2\left \vert \beta \right \vert }{n}}.
\end{equation*}
\end{proof}

\begin{lemma}
\textit{Let} $V\in S_{p}^{\beta }$. \textit{Then th}ere \textit{exists} a 
\textit{positive constant} $C_{d}=C_{d}\left( n\right) $ \textit{such that} 
\begin{equation*}
\eta _{\beta }\left( r\right) \leq C_{d}\eta _{\beta }(\frac{r}{2})\ ,\ r>0.
\end{equation*}
\end{lemma}

\begin{proof}
Let $m=m\left( n\right) \in \mathbb{N},$ $x_{1},\ldots ,x_{m\left( n\right)
}\in B_{\beta }\left( x_{0},r\right) $ such that 
\begin{equation*}
B_{\beta }\left( x_{0},r\right) \subseteq \bigcup_{j=1}^{m}B_{\beta }\left(
x_{j},\frac{r}{2}\right) .
\end{equation*}%
We have%
\begin{equation*}
\int_{\left\vert x_{0}-y\right\vert _{\beta }<r}\frac{\left\vert V\left(
y\right) \right\vert }{\left\vert x_{0}-y\right\vert _{\beta }^{\left(
n-p\right) \frac{2\left\vert \beta \right\vert }{n}}}dy\leq
\sum_{j=1}^{m}\int_{\left\vert x_{j}-y\right\vert _{\beta }<\frac{r}{2}}%
\frac{\left\vert V\left( y\right) \right\vert }{\left\vert
x_{0}-y\right\vert _{\beta }^{\left( n-p\right) \frac{2\left\vert \beta
\right\vert }{n}}}dy=\sum_{j=1}^{m}I_{j}
\end{equation*}%
and%
\begin{equation*}
\begin{array}{ll}
I_{j} & =\int_{\left\vert x_{0}-y\right\vert _{\beta }\geq \left\vert
x_{j}-y\right\vert _{\beta },\left\vert x_{j}-y\right\vert _{\beta }<\frac{r%
}{2}}\frac{\left\vert V\left( y\right) \right\vert }{\left\vert
x_{0}-y\right\vert _{\beta }^{\left( n-p\right) \frac{2\left\vert \beta
\right\vert }{n}}}dy \\ 
&  \\ 
& +\int_{\left\vert x_{0}-y\right\vert _{\beta }<\left\vert
x_{j}-y\right\vert _{\beta }<\frac{r}{2}}\frac{\left\vert V\left( y\right)
\right\vert }{\left\vert x_{0}-y\right\vert _{\beta }^{\left( n-p\right) 
\frac{2\left\vert \beta \right\vert }{n}}}dy \\ 
&  \\ 
& =A_{j}+B_{j}.%
\end{array}%
\end{equation*}%
Since 
\begin{equation*}
A_{j}\leq \int_{\left\vert x_{j}-y\right\vert _{\beta }<\frac{r}{2}}\frac{%
\left\vert V\left( y\right) \right\vert }{\left\vert x_{j}-y\right\vert
_{\beta }^{\left( n-p\right) \frac{2\left\vert \beta \right\vert }{n}}}%
dy\leq \eta _{\beta }(\frac{r}{2})
\end{equation*}%
\begin{equation*}
B_{j}\leq \int_{\left\vert x_{0}-y\right\vert _{\beta }<\frac{r}{2}}\frac{%
\left\vert V\left( y\right) \right\vert }{\left\vert x_{0}-y\right\vert
_{\beta }^{\left( n-p\right) \frac{2\left\vert \beta \right\vert }{n}}}%
dy\leq \eta _{\beta }(\frac{r}{2})
\end{equation*}%
then, we get the conclusion.
\end{proof}

The following definition gives a generalization of $S_{p}^{\beta }$.

\begin{definition}
Let $\varphi :\left] 0,+\infty \right[ \rightarrow \left] 0,+\infty \right[ $
be a non-decreasing continuous function with $\lim\limits_{t\rightarrow
0}\varphi \left( t\right) =0$. We say that $V:\mathbb{R}^{n}\rightarrow 
\mathbb{R}$ belongs to the class $S_{p,\varphi }^{\beta }$ if and only if
there exists a non decreasing function $\xi _{\beta }:\left] 0,+\infty %
\right[ \rightarrow \left] 0,+\infty \right[ $ with $\lim\limits_{r%
\rightarrow 0}\xi _{\beta }\left( r\right) =0$ such that%
\begin{equation}
\sup_{x\in \mathbb{R}^{n}}\int_{\left\vert x-y\right\vert _{\beta }<r}\frac{%
\left\vert V\left( y\right) \right\vert }{\left\vert x-y\right\vert _{\beta
}^{\left( n-p\right) \frac{2\left\vert \beta \right\vert }{n}}\varphi \left(
\left\vert x-y\right\vert _{\beta }\right) }dy\leq \xi _{\beta }\left(
r\right) ,1<p<n.  \label{1.1}
\end{equation}
\end{definition}

In order to show that a function $V\in S_{p}^{\beta }$ belongs to an
appropriate $S_{p,\varphi }^{\beta }$ we give the following lemma.

\begin{lemma}
\textit{Let} $V\in S_{p}^{\beta }$ \textit{such that} $\exists \gamma \in %
\left] 0,1\right[ :$ $\int_{0}^{1}t^{-1}\eta _{\beta }^{1-\gamma }\left(
t\right) dt<+\infty $, \textit{where} $\eta _{\beta }\left( t\right) $ 
\textit{is the Stummel modulus generated by }$\beta $-distance \textit{of} $%
V $. \textit{Then} $V\in S_{p,\eta _{\beta }^{\gamma }}^{\beta }$ \textit{and%
}%
\begin{equation}
\int_{\left \vert x-y\right \vert _{\beta }<r}\frac{\left \vert V\left(
y\right) \right \vert }{\left \vert x-y\right \vert _{\beta }^{\left(
n-p\right) \frac{2\left \vert \beta \right \vert }{n}}\eta _{\beta }^{\gamma
}\left( \left \vert x-y\right \vert _{\beta }\right) }dy\leq \mu _{\beta
}\left( r\right) ,  \label{1.2}
\end{equation}%
\textit{wh}ere 
\begin{equation*}
\mu _{\beta }\left( r\right) =\frac{2}{C}\int_{0}^{r}t^{-1}\eta _{\beta
}^{1-\gamma }\left( t\right) dt.
\end{equation*}

\begin{proof}
Using Lemma 2, we can obtain%
\begin{equation*}
\begin{array}{ll}
\int_{\left \vert x-y\right \vert _{\beta }<r}\frac{\left \vert V\left(
y\right) \right \vert }{\left \vert x-y\right \vert _{\beta }^{\left(
n-p\right) \frac{2\left \vert \beta \right \vert }{n}}\eta _{\beta }^{\gamma
}\left( \left \vert x-y\right \vert _{\beta }\right) }dy & =\sum
\limits_{k=0}^{+\infty }\int_{\frac{r}{2^{k+1}}\leq \left \vert x-y\right
\vert _{\beta }<\frac{r}{2^{k}}}\frac{\left \vert V\left( y\right) \right
\vert }{\left \vert x-y\right \vert _{\beta }^{\left( n-p\right) \frac{%
2\left \vert \beta \right \vert }{n}}\eta _{\beta }^{\gamma }\left( \left
\vert x-y\right \vert _{\beta }\right) }dy \\ 
&  \\ 
& \leq \sum \limits_{k=0}^{+\infty }\eta _{\beta }\left( \frac{r}{2^{k}}%
\right) \left[ \eta _{\beta }\left( \frac{r}{2^{k+1}}\right) \right]
^{-\gamma } \\ 
&  \\ 
& \leq C^{-\gamma }\sum \limits_{k=0}^{+\infty }\left[ \eta _{\beta }\left( 
\frac{r}{2^{k}}\right) \right] ^{1-\gamma }.%
\end{array}%
\end{equation*}%
The last series converges observing that%
\begin{equation*}
\begin{array}{ll}
\int_{0}^{r}t^{-1}\eta _{\beta }^{1-\gamma }\left( t\right) dt & =\sum
\limits_{k=0}^{+\infty }\int_{\frac{r}{2^{k+1}}}^{\frac{r}{2^{k}}}t^{-1}\eta
_{\beta }^{1-\gamma }\left( t\right) dt \\ 
&  \\ 
& \geq \sum \limits_{k=0}^{+\infty }\left[ \eta _{\beta }\left( \frac{r}{%
2^{k+1}}\right) \right] ^{1-\gamma }\frac{2^{k}}{r}\int \limits_{\frac{r}{%
2^{k+1}}}^{\frac{r}{2^{k}}}dt \\ 
&  \\ 
& \geq \frac{1}{2}C^{1-\gamma }\sum \limits_{k=0}^{+\infty }\left[ \eta
_{\beta }\left( \frac{r}{2^{k}}\right) \right] ^{1-\gamma }.%
\end{array}%
\end{equation*}
\end{proof}
\end{lemma}

\section{MAIN RESULTS}

In this section ,under the more general hypotheses for function $V$, we will
obtain embeddings like (\ref{0.1}) using $\beta $-distance.

Firstly we need the following definitions:

Let $f$ and $h$ be measurable functions such that $f\in L_{1}^{loc}\left( 
\mathbb{R}^{n}\right) $ and $h\geq 0$, we set the fractional integral
generated by $\beta $-distance of order $p$ as%
\begin{equation}
I_{p}^{\beta }\left( f\right) \left( x\right) =\int_{\mathbb{R}^{n}}\frac{%
\left\vert f\left( y\right) \right\vert }{\left\vert x-y\right\vert _{\beta
}^{\left( n-p\right) \frac{2\left\vert \beta \right\vert }{n}}}dy,  \label{2}
\end{equation}%
and we get generalized fractional integral generated by $\beta $-distance;%
\begin{equation}
I_{p,h}^{\beta }\left( f\right) \left( x\right) =\int_{\mathbb{R}^{n}}\frac{%
\left\vert f\left( y\right) \right\vert }{\left\vert x-y\right\vert _{\beta
}^{\left( n-p\right) \frac{2\left\vert \beta \right\vert }{n}}h\left(
\left\vert x-y\right\vert _{\beta }\right) }dy.  \label{2.0}
\end{equation}%
The important properties of the fractional integrals, their generalizations
were studied by many authors. We refer to papers \cite{LE64}$-$\cite{OZ06},\ 
\cite{YI05}.

\begin{theorem}
\textit{Let} $V\in S_{p,\varphi }^{\beta }$ \textit{with} $\varphi \left(
t\right) $ \textit{and} $\xi _{\beta }\left( t\right) $ as \textit{in
Definition} 2. \textit{Then, for any} $\sigma \in \left] 0,1\right[ $, 
\textit{there exists a non-decreasing, positive function} $G\left( t\right) $
\textit{such that}%
\begin{equation}
\int_{B_{\beta }\left( y,r\right) }G\left( \frac{I_{p,\varphi ^{\sigma
}}^{\beta }\left( f^{p}\right) }{\left\Vert f\right\Vert _{p}^{p}}\right)
V\left( x\right) dx\leq \xi _{\beta }\left( r\right)   \label{2.1}
\end{equation}%
\textit{for all} $f\in C_{0}^{\infty }\left( \mathbb{R}^{n}\right) $ ,%
\textit{where} $B_{\beta }\left( .,r\right) $ \textit{is } $\beta -$ball 
\textit{with radius} $r$ \textit{containing the support of} $f.$ \textit{Also%
}%
\begin{equation}
\lim_{t\rightarrow \infty }\frac{G\left( t\right) }{t}=+\infty .  \label{2.2}
\end{equation}
\end{theorem}

\begin{proof}
For $\epsilon >0$ and $0<\sigma <1$, we obtain,%
\begin{equation}
\begin{array}{ll}
I_{p,\varphi ^{\sigma }}^{\beta }\left( f^{p}\right) \left( x\right) & 
=\int_{\left \vert x-y\right \vert _{\beta }\leq \varepsilon }\frac{\left
\vert f\left( y\right) \right \vert ^{p}\varphi \left( \left \vert x-y\right
\vert _{\beta }\right) }{\left \vert x-y\right \vert _{\beta }^{\left(
n-p\right) \frac{2\left \vert \beta \right \vert }{n}}\varphi ^{\sigma
}\left( \left \vert x-y\right \vert _{\beta }\right) \varphi \left( \left
\vert x-y\right \vert _{\beta }\right) }dy \\ 
&  \\ 
& +\int_{\left \vert x-y\right \vert _{\beta }>\varepsilon }\frac{\left
\vert f\left( y\right) \right \vert ^{p}}{\left \vert x-y\right \vert
_{\beta }^{\left( n-p\right) \frac{2\left \vert \beta \right \vert }{n}%
}\varphi ^{\sigma }\left( \left \vert x-y\right \vert _{\beta }\right) }dy
\\ 
&  \\ 
& \leq \varphi ^{1-\sigma }\left( \varepsilon \right) I_{p,\varphi }^{\beta
}\left( f^{p}\right) +\frac{1}{\epsilon ^{\left( n-p\right) \frac{2\left
\vert \beta \right \vert }{n}}\varphi ^{\sigma }\left( \varepsilon \right) }%
\Vert f\Vert _{p}^{p}.%
\end{array}
\label{2.3}
\end{equation}%
Letting $\varepsilon ^{\left( n-p\right) \frac{2\left \vert \beta
\right
\vert }{n}}\varphi \left( \varepsilon \right) =\Phi \left(
\varepsilon \right) $ , we choose 
\begin{equation*}
\varepsilon =\Phi ^{-1}\left( \frac{||f||_{p}^{p}}{I_{p,\varphi }^{\beta
}\left( f^{p}\right) }\right) \ ,
\end{equation*}%
a choice which makes the two terms on the right hand side of (\ref{2.3})
equal.

From (\ref{2.3}), we obtain 
\begin{equation*}
\frac{I_{p,\varphi ^{\sigma }}^{\beta }f^{p}\left( x\right) }{||f||_{p}^{p}}%
\leq \frac{2}{\left[ \Phi ^{-1}\left( \frac{||f\Vert _{p}^{p}}{I_{p,\varphi
}^{\beta }\left( f^{p}\right) }\right) \right] ^{\left( n-p\right) \frac{%
2\left \vert \beta \right \vert }{n}}\varphi ^{\sigma }\left[ \Phi
^{-1}\left( \frac{||f\Vert _{p}^{p}}{I_{p,\varphi }^{\beta }\left(
f^{p}\right) }\right) \right] }.
\end{equation*}%
If 
\begin{equation*}
\psi \left( t\right) =\frac{2}{\left[ \Phi ^{-1}(\frac{1}{t})\right]
^{\left( n-p\right) \frac{2\left \vert \beta \right \vert }{n}}\varphi
^{\sigma }\left[ \Phi ^{-1}\left( \frac{1}{t}\right) \right] }
\end{equation*}%
and 
\begin{equation*}
G\left( t\right) =\psi ^{-1}\left( t\right) ,
\end{equation*}%
we have%
\begin{equation*}
G\left( \frac{I_{p,\varphi ^{\sigma }}^{\beta }\left( f^{p}\right) }{%
||f||_{p}^{p}}\right) \leq \frac{I_{p,\varphi }^{\beta }\left( f^{p}\right) 
}{||f||_{p}^{p}}.
\end{equation*}%
Finally, using Fubini's theorem 
\begin{equation*}
\begin{array}{l}
\int_{B_{\beta }\left( y,r\right) }G\left( \frac{I_{p,\varphi ^{\sigma
}}^{\beta }\left( f^{p}\right) \left( x\right) }{\Vert f\Vert _{p}^{p}}%
\right) |V\left( x\right) |dx \\ 
\\ 
\leq \frac{1}{\Vert f\Vert _{p}^{p}}\int_{B_{\beta }\left( y,r\right)
}I_{p,\varphi }^{\beta }\left( f^{p}\right) \left( x\right) |V\left(
x\right) |dx \\ 
\\ 
=\frac{1}{\Vert f\Vert _{p}^{p}}\int_{B_{\beta }\left( y,r\right) }\left(
\int_{\mathbb{R}^{n}}\frac{|f\left( y\right) |^{p}}{\left \vert x-y\right
\vert _{\beta }^{\left( n-p\right) \frac{2\left \vert \beta \right \vert }{n}%
}\varphi \left( \left \vert x-y\right \vert _{\beta }\right) }dy\right)
\left \vert V\left( x\right) \right \vert dx \\ 
\\ 
=\frac{1}{\Vert f\Vert _{p}^{p}}\int_{\mathbb{R}^{n}}\left( \int_{B_{\beta
}\left( y,r\right) }\frac{|V\left( x\right) |}{\left \vert x-y\right \vert
_{\beta }^{\left( n-p\right) \frac{2\left \vert \beta \right \vert }{n}%
}\varphi \left( \left \vert x-y\right \vert _{\beta }\right) }dx\right)
\left \vert f\left( y\right) \right \vert ^{p}dy\leq \xi _{\beta }\left(
r\right) .%
\end{array}%
\end{equation*}%
So (\ref{2.1}) was obtained.

(\ref{2.2}) is easily seen to be equivalent to%
\begin{equation}
\lim_{s\rightarrow 0}\frac{[\Phi ^{-1}\left( s\right) ]^{\left( n-p\right) 
\frac{2\left \vert \beta \right \vert }{n}}\varphi ^{\sigma }[\Phi
^{-1}\left( s\right) ]}{s}=+\infty .  \label{2.4}
\end{equation}%
Choosing $H\left( t\right) =t^{\left( n-p\right) \frac{2\left \vert \beta
\right \vert }{n}}\varphi ^{\sigma }\left( t\right) ,(2.4)$ can be rewritten
as%
\begin{equation}
\lim_{s\rightarrow 0}\frac{H\left( \Phi ^{-1}\left( s\right) \right) }{s}%
=+\infty .  \label{2.5}
\end{equation}%
Since $\lim_{s\rightarrow 0}\frac{\Phi \left( s\right) }{H\left( s\right) }%
=\lim_{s\rightarrow 0}\varphi ^{1-\sigma }\left( s\right) =0$ we obtain (\ref%
{2.5}).
\end{proof}

\begin{lemma}
\textit{Let} $h:\left] 0,+\infty \right[ \rightarrow \left] 0,+\infty \right[
$ \textit{such that} $\int_{0}^{1}\frac{[h\left( t\right) ]^{p^{\prime }/p}}{%
t}dt<+\infty $ $(p^{\prime }:\frac{1}{p^{\prime }}+\frac{1}{p}=1)$. \textit{%
Then} 
\begin{equation*}
I_{1}^{\beta }\left( f\right) \leq C\left( n,p,diam\left( sptf\right)
,h\right) \left[ I_{p,h}\left( f^{p}\right) \right] ^{\frac{1}{p}}
\end{equation*}%
\textit{for all} $f\in C_{0}^{\infty }\left( \mathbb{R}^{n}\right) $.
\end{lemma}

\begin{proof}
Using H\"{o}lder inequality, we get%
\begin{equation*}
\begin{array}{ll}
I_{1}^{\beta }\left( f\right) & =\int_{\mathbb{R}^{n}}\frac{|f\left(
y\right) |h^{\frac{1}{p}}\left( \left \vert x-y\right \vert _{\beta }\right) 
}{\left \vert x-y\right \vert _{\beta }^{n-1}h^{\frac{1}{p}}\left( \left
\vert x-y\right \vert _{\beta }\right) }dy \\ 
& \leq \left[ I_{p,h}\left( f^{p}\right) \right] ^{\frac{1}{p}}\left(
\int_{B_{\beta }\left( y,r\right) }\frac{h^{\frac{p^{\prime }}{p}}\left(
\left \vert x-y\right \vert _{\beta }\right) }{\left \vert x-y\right \vert
_{\beta }^{n}}dy\right) ^{\frac{1}{p^{\prime }}}%
\end{array}%
\end{equation*}%
where $B_{\beta }\left( y,r\right) \supseteq $ $sptf.$
\end{proof}

\begin{corollary}
\textit{Under the} \textit{hypotheses of Theorem} 1 \textit{and} \textit{for
all} $u\in C_{0}^{\infty }\left( \mathbb{R}^{n}\right) $, lett\textit{ing} $%
\int_{0}^{1}\frac{[\varphi \left( t\right) ]^{\frac{\sigma p^{\prime }}{p}}}{%
t}dt<+\infty $, \textit{w}e \textit{get}%
\begin{equation}
\int_{B_{\beta }\left( y,r\right) }G\left( \frac{|u|^{p}}{\left \Vert \nabla
u\right \Vert _{p}^{p}}\right) V\left( x\right) dx\leq C\left(
n,p,diam\left( sptu\right) ,\varphi \right) \xi _{\beta }\left( r\right) ,
\label{2.6}
\end{equation}%
\textit{where} $B_{\beta }\left( y,r\right) \supseteq $ $sptu$.
\end{corollary}

\begin{proof}
Using Lemma 1 and Theorem 1, we have the following inequality%
\begin{equation*}
|u|\leq C\left( n\right) I_{1}^{\beta }\left( \left \vert \nabla u\right
\vert \right) .
\end{equation*}
\end{proof}

\begin{remark}
If we choose the function $\varphi ^{\sigma }\left( t\right) $ with a more
general non-decreasing function $\delta :\left] 0,+\infty \right[
\rightarrow \left] 0,+\infty \right[ $ such that%
\begin{equation*}
\begin{array}{cc}
\lim\limits_{t\rightarrow 0}\delta \left( t\right) =0, & \lim\limits_{t%
\rightarrow 0}\frac{\varphi \left( t\right) }{\delta \left( t\right) }=0,%
\end{array}%
\end{equation*}%
$\frac{\varphi \left( t\right) }{\delta \left( t\right) }$ is
non-decreasing, where $\varphi \left( t\right) $ is as in Definition 2, the
previous results are also valid.
\end{remark}

\begin{proposition}
\textit{Let} $V\in S_{p}^{\beta },V\geq 0,\sigma \in \left] 0,1\right[ $, $%
\gamma =\frac{1}{\frac{\sigma p}{p}+1}$ \textit{and} assume that%
\begin{equation}
\int\limits_{0}^{1}\frac{\left[ \eta _{\beta }\left( t\right) \right]
^{1-\gamma }}{t}dt<+\infty   \label{2.7}
\end{equation}
\end{proposition}

\textit{Then}%
\begin{equation}
V\in S_{p,\eta _{\beta }^{\gamma }}  \label{2.8}
\end{equation}%
\textit{and for} ev\textit{ery} $u\in C_{0}^{\infty }\left( \mathbb{R}%
^{n}\right) $, \textit{th}ere \textit{exists} a \textit{non decreasing
positive function} $G\left( t\right) $ \textit{such that}%
\begin{equation}
\int_{B_{\beta }\left( y,r\right) }G\left( \frac{\left \vert u\right \vert
^{p}}{\left \Vert \nabla u\right \Vert _{p}^{p}}\right) V\left( x\right)
dx\leq C\left( n,p,\eta _{\beta }\right) \mu _{\beta }\left( r\right)
\label{2.9}
\end{equation}%
\textit{where} $B_{\beta }\left( y,r\right) \supseteq $ spt$u$ and%
\begin{equation}
\mu _{\beta }\left( r\right) =\frac{2}{C}\int \limits_{0}^{r}t^{-1}\eta
_{\beta }^{1-\gamma }\left( t\right) dt.  \label{2.10}
\end{equation}

Now we give an example of a function $f\in S_{p}^{\beta }$, $f\geq 0$. But
for $\lambda >n-2$, we choose $f\not \in L_{1,\lambda }^{\beta }.$

\begin{example}
Let $\chi _{B}\left( y\right) $ is the characteristic function of $B$ and 
\begin{equation*}
f\left( x\right) =\frac{1}{\left \vert x\right \vert _{\beta }^{2}\left
\vert \log \left \vert x\right \vert _{\beta }\right \vert ^{6}}\chi
_{B}\left( x\right) \ ,
\end{equation*}%
where, $B_{\beta }\left( 0,\delta \right) $ the $\beta -$ball centered in $0$
and radius $\delta =e^{-3}$. We obtain that the function 
\begin{equation*}
\eta _{\beta }\left( r\right) =\sup_{x\in \mathbb{R}^{n}}\int_{\left \vert
x-y\right \vert _{\beta }<r}\frac{f\left( y\right) }{\left \vert x-y\right
\vert _{\beta }^{n-2}}dy,
\end{equation*}%
is such that
\end{example}

(i) $\lim_{r\rightarrow 0}\eta _{\beta }\left( r\right) =0$

(ii) $\int_{0}^{r}\frac{\eta _{\beta }^{1/4}(\rho )}{\rho }d\rho <+\infty .$

\begin{proof}
For $x\in \mathbb{R}^{n}$ and $r>0$ we obtain%
\begin{equation*}
\int_{\left \vert x-y\right \vert _{\beta }<r}\frac{1}{\left \vert y\right
\vert _{\beta }^{2}\left \vert x-y\right \vert _{\beta }^{n-2}\left \vert
\log \left \vert y\right \vert _{\beta }\right \vert ^{6}}\chi _{B}\left(
y\right) dy
\end{equation*}%
\begin{equation*}
\begin{array}{l}
=\dint_{\left \vert y\right \vert _{\beta }<\left \vert x-y\right \vert
_{\beta }<r}\frac{1}{\left \vert y\right \vert _{\beta }^{2}\left \vert
x-y\right \vert _{\beta }^{n-2}\left \vert \log \left \vert y\right \vert
_{\beta }\right \vert ^{6}}\chi _{B}\left( y\right) dy \\ 
+\dint_{\left \{ \left \vert x-y\right \vert _{\beta }<r\right \} \cap \left
\{ \left \vert x-y\right \vert _{\beta }<\left \vert y\right \vert _{\beta
}<\delta \right \} }\frac{1}{\left \vert y\right \vert _{\beta }^{2}\left
\vert x-y\right \vert _{\beta }^{n-2}\left \vert \log \left \vert y\right
\vert _{\beta }\right \vert ^{6}}\chi _{B}\left( y\right) dy=A_{1}+A_{2}.%
\end{array}%
\end{equation*}%
For $A_{1},$letting $\sigma =\min \left( r,\delta \right) $%
\begin{equation*}
\begin{array}{ll}
A_{1} & =\dint_{\left \vert y\right \vert _{\beta }<\left \vert x-y\right
\vert _{\beta }<r}\frac{1}{\left \vert y\right \vert _{\beta }^{2}\left
\vert x-y\right \vert _{\beta }^{n-2}\left \vert \log \left \vert y\right
\vert _{\beta }\right \vert ^{6}}\chi _{B}\left( y\right) dy \\ 
&  \\ 
& \leq \dint_{\left \vert y\right \vert _{\beta }<r}\frac{1}{\left \vert
y\right \vert _{\beta }^{n}\left \vert \log \left \vert y\right \vert
_{\beta }\right \vert ^{6}}\chi _{B}\left( y\right) dy=C\left( n\right) 
\frac{1}{(-\log \sigma )^{5}},%
\end{array}%
\end{equation*}%
and for $A_{2}$, considering that the function $\frac{1}{t^{2}(-\log t)^{6}}$
is decreasing in $\left] 0,e^{-3}\right[ $, we obtain ;%
\begin{equation*}
\begin{array}{ll}
A_{2} & =\dint_{\left \{ \left \vert x-y\right \vert _{\beta }<r\right \}
\cap \left \{ \left \vert x-y\right \vert _{\beta }<\left \vert y\right
\vert _{\beta }<\delta \right \} }\frac{1}{\left \vert y\right \vert _{\beta
}^{2}\left \vert x-y\right \vert _{\beta }^{n-2}\left \vert \log \left \vert
y\right \vert _{\beta }\right \vert ^{6}}\chi _{B}\left( y\right) dy \\ 
&  \\ 
& =\dint_{\left \{ \left \vert x-y\right \vert _{\beta }<r\right \} \cap
\left \{ \left \vert x-y\right \vert _{\beta }<\left \vert y\right \vert
_{\beta }<\delta \right \} }\frac{1}{\left \vert y\right \vert _{\beta
}^{2}\left \vert x-y\right \vert _{\beta }^{n-2}\left \vert \log \left \vert
y\right \vert _{\beta }\right \vert ^{6}}dy \\ 
&  \\ 
& \leq \dint_{\left \{ \left \vert z\right \vert _{\beta }<r\right \} \cap
\left \{ \left \vert z\right \vert _{\beta }<\delta \right \} }\frac{dz}{%
\left \vert z\right \vert _{\beta }^{n}(-\log \left \vert z\right \vert
_{\beta })^{6}}=C\left( n\right) \frac{1}{\left( -\log \sigma \right) ^{5}}.%
\end{array}%
\end{equation*}%
Then we have%
\begin{equation*}
\eta _{\beta }\left( r\right) \leq L\left( r\right) \equiv 2C\left( n\right) 
\frac{1}{\left( -\log \sigma \right) ^{5}}.
\end{equation*}%
Because $\lim \limits_{r\rightarrow 0}L\left( r\right) =0$ we get (i).

Only considering $r<\delta ,$%
\begin{equation*}
\begin{array}{ll}
\dint_{0}^{r}\frac{\eta _{\beta }^{\frac{1}{4}}(\rho )}{\rho }d\rho & \leq
\dint_{0}^{r}\frac{L^{\frac{1}{4}}\left( \rho \right) }{\rho }d\rho \\ 
&  \\ 
& =\left( 2C\left( n\right) \right) ^{\frac{1}{4}}\dint_{0}^{r}\frac{\left(
-\log \rho \right) ^{-\frac{5}{4}}}{\rho }d\rho \\ 
&  \\ 
& =(2C\left( n\right) )^{\frac{1}{4}}\frac{4}{\left( -\log r\right) ^{\frac{1%
}{4}}}<+\infty .%
\end{array}%
\end{equation*}%
So we prove (ii).

Now, for $\lambda >n-2$, we prove that the function $f\not\in L_{1,\lambda
}^{\beta }.$

Indeed letting, for $\varepsilon >0,\lambda =n-2+\varepsilon $, the
following quantity is unbounded.%
\begin{equation*}
\begin{array}{ll}
\frac{1}{r^{n-2+\varepsilon }}\dint_{B_{\beta }\left( 0,r\right) }\frac{\chi
_{B}\left( y\right) }{\left\vert y\right\vert _{\beta }^{2}\left\vert \log
\left\vert y\right\vert _{\beta }\right\vert ^{6}}dy & =\frac{C\left(
n\right) }{r^{n-2+\varepsilon }}\dint_{0}^{r}\frac{\rho ^{n-1}}{\rho
^{2}(-\log \rho )^{6}}d\rho \\ 
&  \\ 
& >\frac{C\left( n\right) }{2^{n-2}r^{\varepsilon }}\dint_{\frac{r}{2}}^{r}%
\frac{d\rho }{(-\log \rho )^{6}\rho } \\ 
&  \\ 
& =\frac{1}{5}\frac{C\left( n\right) }{2^{n-2}r^{\varepsilon }}\left[ \frac{1%
}{(-\log r)^{5}}-\frac{1}{(-\log (\frac{1}{2}r))^{5}}\right] .%
\end{array}%
\end{equation*}
\end{proof}

\begin{remark}
Throughout this study, if we choose $\beta _{1}=\beta _{2}=...=\beta _{n}=%
\frac{1}{2},\ $then we have the conclusions of \cite{RaCaZa01}.\newpage
\end{remark}


\begin{thebibliography}{99}
\bibitem{AD71} D. Adams: Traces of potentials arising from traslation
invariant operators. Ann. Scuola Norm. Sup. Pisa 25 (1971), 203-217.

\bibitem{BeLi60} O.V. Besov, P.I. Lizorkin, The $L^{p}$ estimates of a
certain class of non-isotropic singular integrals, Dokl. Akad. Nauk, SSSR,
69 (1960), 1250-1253.

\bibitem{GaCuMa01} J. Garcia-Cuerva, J.M. Martell, Two-weight norm
inequalies for maximal operator and fractional integrals on non-homogeneous
spaces, Indiana Univ. Math. J., 50, No. 3 (2001), 1241-1280.

\bibitem{HE72} L. Hedberg: On certain convolution inequalities. Proc. Amer.
Math. Soc. 36 (1972), 505-510.

\bibitem{Kuf} A. Kufner, O. John and S. Fucik, \textit{Function spaces},
Noordhoff, Leyden and Academia, Prague, 1977.

\bibitem{LE64} B.M. Levitan, Generelized Translation Operators and Some of
Their Applications, Moscow (1962), Translation 1964.

\bibitem{Morrey} C.B. Morrey, \textit{On the solutions of quasi-linear
elliptic partial differential equations,} Trans. Amer. Math. Soc. \textbf{43}
(1938), 126-166.

\bibitem{RaCaZa01} M. A. Ragusa, Catania, and P. Zamboni,
Sant'agata-Messina, A Potential Theoretic Inequality, Czechoslovak
Mathematical Journal, 51 (126) (2001), 55-65

\bibitem{SKM93} S.G. Samko, A.A. Kilbas, and O.I. Marichev, Fractional
Integrals and Derivatives - Theory and Applications, Gordon and Breach,
Linghorne, 1993.

\bibitem{SY06} M.Z. Sarikaya, H. Y\i ld\i r\i m, The restriction and the
continuity properties of potentials depending on $\lambda $-distance, Turk.
J. Math., 30, No. 3 (2006).

\bibitem{SaYý006} M.Z. Sarikaya, H. Y\i ld\i r\i m, On the $\beta $%
-spherical Riesz potential generated by the $\beta $-distance, Int. Journal
of Contemp. Math. Sciences, 1, No. 1-4 (2006), 85-89.

\bibitem{SaYý0006} M.Z. Sarikaya, H. Y\i ld\i r\i m$.$, On the non-isotropic
fractional integrals generated by the $\lambda $-distance, Sel\c{c}uk
Journal of Appl. Math., 1 (2006).

\bibitem{OZ06} M.Z. Sarikaya, H. Y\i ld\i r\i m, U. M. Ozkan, Norm
inequalities with non-isotropic kernels, Int. Journal of Pure and Applied
Mathematics, 31, No. 3 (2006).

\bibitem{SC86} M. Schechter: Spectra of Partial Differential Operators.
North Holland, 1986.

\bibitem{ST70} E.M. Stein, Singular Integrals Differential Properties of
Functions, Princeton Uni. Press, Princeton, New Jersey (1970).

\bibitem{WE75} G. V. Welland: Weighted norm inequalities for fractional
integral. Proc. Amer. Math. Soc. 51 (1975), 143-148.

\bibitem{YI05} H. Y\i ld\i r\i m, On generalization of the quasi homogeneous
Riesz potential, Turk. J. Math., 29 (2005), 381-387.
\end{thebibliography}
\end{document}